\documentclass[oneside]{amsart}

\usepackage{amsthm}
\usepackage{amsmath}
\usepackage{amssymb}
\usepackage{enumerate}
\usepackage{graphicx}
\usepackage[hidelinks,pagebackref,pdftex]{hyperref}
\usepackage{booktabs}
\usepackage{color}
\usepackage[dvipsnames]{xcolor}
\usepackage{import}
\usepackage{tikz-cd}

\AtBeginDocument{%
   \def\MR#1{}
}

\usepackage{marginnote}
\long\def\@savemarbox#1#2{\global\setbox#1\vtop{\hsize\marginparwidth 
  \@parboxrestore\tiny\raggedright #2}}
\renewcommand*{\backref}[1]{}
\renewcommand*{\backrefalt}[4]{
  \ifcase #1
  [No citations.]
  \or [#2]
  \else [#2]
  \fi }

\numberwithin{equation}{section}
\theoremstyle{plain}
\newtheorem{theorem}[equation]{Theorem}

\newtheorem{lemma}[equation]{Lemma}

\newtheorem{proposition}[equation]{Proposition}

\newtheorem{claim}[equation]{Claim}

\newtheorem*{namedtheorem}{\theoremname}
\newcommand{\theoremname}{testing}

\theoremstyle{definition}
\newtheorem{definition}[equation]{Definition}
\newtheorem{remark}[equation]{Remark}

\newcommand{\cut}{\backslash\backslash}

\newcommand{\fakeenv}{} 

{
 \renewcommand{\fakeenv}{#2} 
 \theoremstyle{plain}
 \newtheorem*{\fakeenv}{#1~\ref{#2}} 
 \begin{\fakeenv}
}
{
 \end{\fakeenv}
}

\makeatletter

\def\chaptermark#1{}

\def\chapter{%
  \if@openright\cleardoublepage\else\clearpage\fi
  \thispagestyle{plain}\global\@topnum\z@
  \@afterindenttrue \secdef\@chapter\@schapter}

\def\@chapter[#1]#2{\refstepcounter{chapter}%
  \ifnum\c@secnumdepth<\z@ \let\@secnumber\@empty
  \else \let\@secnumber\thechapter \fi
  \typeout{\chaptername\space\@secnumber}%
  \def\@toclevel{0}%
  \ifx\chaptername\appendixname \@tocwriteb\tocappendix{chapter}{#2}%
  \else \@tocwriteb\tocchapter{chapter}{#2}\fi
  \chaptermark{#1}%
  \addtocontents{lof}{\protect\addvspace{10\p@}}%
  \addtocontents{lot}{\protect\addvspace{10\p@}}%
  \@makechapterhead{#2}\@afterheading}
\def\@schapter#1{\typeout{#1}%
  \let\@secnumber\@empty
  \def\@toclevel{0}%
  \ifx\chaptername\appendixname \@tocwriteb\tocappendix{chapter}{#1}%
  \else \@tocwriteb\tocchapter{chapter}{#1}\fi
  \chaptermark{#1}%
  \addtocontents{lof}{\protect\addvspace{10\p@}}%
  \addtocontents{lot}{\protect\addvspace{10\p@}}%
  \@makeschapterhead{#1}\@afterheading}
\newcommand\chaptername{Chapter}

\def\@makechapterhead#1{\global\topskip 7.5pc\relax
  \begingroup
  \fontsize{\@xivpt}{18}\bfseries\centering
    \ifnum\c@secnumdepth>\m@ne
      \leavevmode \hskip-\leftskip
      \rlap{\vbox to\z@{\vss
          \centerline{\normalsize\mdseries
              \uppercase\@xp{\chaptername}\enspace\thechapter}
          \vskip 3pc}}\hskip\leftskip\fi
     #1\par \endgroup
  \skip@34\p@ \advance\skip@-\normalbaselineskip
  \vskip\skip@ }
\def\@makeschapterhead#1{\global\topskip 7.5pc\relax
  \begingroup
  \fontsize{\@xivpt}{18}\bfseries\centering
  #1\par \endgroup
  \skip@34\p@ \advance\skip@-\normalbaselineskip
  \vskip\skip@ }
\def\appendix{\par
  \c@chapter\z@ \c@section\z@
  \let\chaptername\appendixname
  \def\thechapter{\@Alph\c@chapter}}

\newcounter{chapter}

\newif\if@openright

\makeatother

\title[The stable AC conjecture and thickenable group presentations]{The stable Andrews-Curtis conjecture and \\ thickenable presentations of the trivial group}
\author{Marc Lackenby}
\address{Mathematical Institute, University of Oxford, \newline Woodstock Road, Oxford OX2 6GG, United Kingdom}
\thanks{The author was partially supported by EPSRC grant EP/Y004256/1. For the purpose of open access, the author has applied a CC BY public copyright licence to any author accepted manuscript arising from this submission.}

\begin{document}

\begin{abstract} 
We establish 
an explicit upper bound on the number of stable Andrews-Curtis moves that convert thickenable balanced
presentations of the trivial group to the standard one-generator presentation. We also present a proof that
thickenable balanced presentations of the trivial group satisfy the (unstable) Andrews-Curtis conjecture.
\end{abstract}

\maketitle
\normalsize

\section{Introduction}
The Andrews-Curtis conjecture \cite{AndrewsCurtis} and its relatives are among the most notorious unsolved problems in group theory. The Andrews-Curtis conjecture proposes that any balanced presentation $\langle x_1, \dots, x_n | r_1, \dots, r_n \rangle$ of the trivial group can be transformed into a \emph{standard presentation} $\langle x_1, \dots, x_n | x_1, \dots, x_n \rangle$ via a finite sequence of the following \emph{Andrews-Curtis moves}:
\begin{enumerate}
\item[(0)] remove or introduce $x_i x_i^{-1}$ or $x_i^{-1} x_i$ in some relation $r_j$;
\item[(1)] replace some $r_i$ by $r_i^{-1}$;
\item[(2)] replace some $r_i$ by $r_i r_j$ or $r_j r_i$ for some $j \not= i$;
\item[(3)] replace some $r_i$ by $x_j r_i x_j^{-1}$ or $x_j^{-1} r_i x_j$ for some $j$.
\end{enumerate}
(In this paper, we view the relations in a presentation as elements of the free semi-group generated by the generators and their inverses. The relations are unordered, and so moves that re-order relations are not required.)

In the stable Andrews-Curtis conjecture \cite{HogAngMetzler, Myasnikov, HAMetzler2}, one considers the above moves plus one further modification to a group presentation:
\begin{enumerate}
\item[(4)] add a generator $x_{n+1}$ and a relation $x_{n+1}$, or the reverse of this operation, provided the remaining relations do not include the letters $x_{n+1}$ or $x_{n+1}^{-1}$.
\end{enumerate}
The stable Andrews-Curtis conjecture proposes that any balanced presentation of the trivial group can be reduced to the standard one-generator presentation $\langle x_1 | x_1 \rangle$ using moves (0)-(4), which are called the \emph{stable Andrews-Curtis moves}.

There are many well known potential counterexamples to the stable Andrews-Curtis conjecture, including various infinite families \cite{HogAngMetzler, Myasnikov}. 
We mention two:
\begin{enumerate}
\item the Akbulut-Kirby \cite{AkbulutKirby} examples $\langle x, y | xyx = yxy, x^{k+1} = y^k \rangle$, for $k \geq 3$;
\item the Miller-Schupp \cite{MillerSchupp} examples $\langle x,y | x^{-1} y^2 x = y^3, x=w \rangle$ for any word $w$ with exponent sum $0$ in $x$.
\end{enumerate}

Associated with any finite presentation $P$ of a group, there is a cell complex $K$, called the \emph{presentation 2-complex for $K$}, with a single 0-cell, a 1-cell for each generator, and a 2-cell attached along the path specified by each relation. This has the property that the fundamental group of $K$ is the group specified by $P$. In this paper, we will focus on presentations $P$ that are \emph{thickenable}, which means that $K$ embeds in some 3-manifold. We remark that it is easily checked whether or not a group presentation is thickenable, via an algorithm of Neuwirth \cite{Neuwirth}. 

When examining the Andrews-Curtis conjecture and the stable Andrews-Curtis conjecture, it is natural to define, for a balanced presentation $P$ of the trivial group, the quantities
$AC(P)$ and $SAC(P)$, which are the minimal number of Andrews-Curtis moves (respectively, stable Andrews-Curtis moves) required to transform it to a standard presentation (respectively, standard one-generator presentation). Of course, $AC(P)$ is finite for every such $P$ if and only if the Andrews-Curtis conjecture holds, and a similar statement holds for $SAC(P)$. Furthermore, these quantities are finite and computable for every $P$ if and only if the (stable) Andrews-Curtis conjecture holds and there is an algorithm to detect the trivial group among all balanced group presentations.
Note that it remains a significant unsolved problem (due to Magnus \cite[Problem 1.12]{Kourovka}) whether there is such an algorithm. 

We have very little to say about $AC(P)$, other than that it is finite and computable for thickenable balanced presentations of the trivial group. Instead, our main focus in this paper is on $SAC(P)$.

We say that the \emph{length} of a finite presentation is the sum of the lengths of its relations. This is a natural measure of the complexity of a presentation, particularly because in a balanced presentation of the trivial group, no relation can be the empty word, and hence there are only finitely many balanced presentations of the trivial group with a given length.

Bridson \cite{Bridson} proved the following striking result. Lishak \cite{Lishak} proved a very similar result but with slightly different constants.

\begin{theorem}
\label{Thm:Bridson}
For any given $n \geq 4$ and for infinitely many positive integers $\ell$, there is a balanced presentation $P$ of the trivial group with $n$ generators and $n$ relations, with the following properties:
\begin{enumerate}
\item $P$ can be converted to a standard presentation using Andrews-Curtis moves;
\item the length of $P$ is at most $24(\ell + 1)$;
\item $SAC(P)$ is at least 
\end{enumerate}
$$2 \! \! \! \! \! \! \! \! \! 
\underbrace{
  {{{^{2\vphantom{h}}}^{2\vphantom{h}}}^{\cdots\vphantom{h}}}^{2\vphantom{h}}
}_{\mathrm{height } \log_2(\ell)}
$$
\end{theorem}

By contrast to the above result of Bridson and Lishak, our main theorem shows that thickenable balanced presentations are significantly more tractable. It has been known for some time that thickenable balanced presentations of the trivial group satisfy the stable Andrews-Curtis conjecture. This is a consequence of the Poincar\'e conjecture, now proved by Perelman \cite{Perelman1, Perelman2, Perelman3}. See for example \cite[p.45]{HogAngMetzler}, \cite[p.2296]{Ivanov} and \cite[p.27]{HAMetzler2}. Our main theorem is the following explicit upper bound on the number of stable Andrews-Curtis moves.

\begin{theorem}
\label{Thm:StableACBound}
For any thickenable balanced presentation $P$ of the trivial group with length $\ell$, $SAC(P)$ is at most
$$2^{2^{c \ell^2}}$$
where $c = 2 \cdot 10^6$.
\end{theorem}

Although the above upper bound on the number of stable Andrews-Curtis moves is large, it is significantly smaller than the quantity in Bridson's and Lishak's theorems, because of the fixed height of the tower.

We remark that Theorems \ref{Thm:Bridson} and \ref{Thm:StableACBound} appear to rule out one possible approach to the stable Andrews-Curtis conjecture. Since the stable Andrews-Curtis conjecture holds for thickenable balanced presentations of the trivial group and since the standard one-generator presentation is thickenable, the conjecture is therefore equivalent to the assertion that any balanced presentation of the trivial group can be converted into a thickenable presentation using stable Andrews-Curtis moves. One might hope that there is a simple way to do this. However, the large difference between the bounds in Theorems \ref{Thm:Bridson} and \ref{Thm:StableACBound} highlight that any process that converts a balanced presentation of the trivial group into a thickenable one would require a vast number of stable Andrews-Curtis moves. Hence the prospect of a simple procedure to achieve this seems remote.

Our proof of Theorem \ref{Thm:StableACBound} uses Perelman's solution of the Poincar\'e conjecture \cite{Perelman1, Perelman2, Perelman3} as well as results about triangulations of the 3-sphere. By definition, a thickenable balanced presentation $P$ of the trivial group specifies a 2-complex $K$ that embeds in some 3-manifold. It therefore has a regular neighbourhood $N(K)$ in that 3-manifold. This is homotopy equivalent to the contractible 2-complex $K$, and so by the Poincar\'e conjecture, $N(K)$ is a 3-ball. We first subdivide the presentation complex $K$ so that each 2-cell is a triangle, and then we cap off $\partial N(K)$ with a triangulated 3-ball to form a triangulation of the 3-sphere. It is known that any two triangulations of a compact 3-manifold differ by a sequence of Pachner moves \cite{Pachner}, and in the case of the 3-sphere, there are good bounds on the number of moves that are required, due to Mijatovic \cite{Mijatovic} and King \cite{King}. The challenge that we face is to go from a bound on Pachner moves to a bound on stable Andrews-Curtis moves. This is not straightforward, since Pachner moves will modify the Euler characteristic of the 2-skeleton of the triangulation, whereas stable Andrews-Curtis moves leave the Euler characteristic of the presentation 2-complex unchanged.

It is natural to ask whether the bound in Theorem \ref{Thm:StableACBound} could be improved. The upper bound on $SAC(P)$ is a double exponential function (of $\ell^2$). As we will see, this could be reduced to a single exponential if we allowed a broader set of moves than just the stable Andrews-Curtis moves. Such moves are discussed in Section \ref{Sec:Supplementary}. They include a generalisation of the stable Andrews-Curtis move $(4)$, as well as the possibility of changing the generators of the presentation by a Nielsen automorphism. It is shown that in Section \ref{Sec:Supplementary} that, when applied to balanced presentations of the trivial group, this broader set of moves can be expressed just in terms of the stable Andrews-Curtis moves, but to do so increases the number of moves exponentially. The other reason for the height of the tower is the exponential upper bound on Pachner moves for triangulations of the 3-sphere proved by Mijatovic \cite{Mijatovic} and King \cite{King}. It remains a possibility that one can convert any triangulation of the 3-sphere with $t$ tetrahedra to a standard triangulation (with just two tetrahedra) via Pachner moves, where the number of such moves is at most a polynomial function of $t$. If true, this would reduce the bound on $SAC(P)$ in Theorem \ref{Thm:StableACBound} to just a single exponential function of some polynomial in $\ell$.

Although we focus, in this paper, mostly on the stable Andrews-Curtis conjecture, we also observe the following.

\begin{theorem}
\label{Thm:AC}
Any thickenable balanced presentation of the trivial group can be converted to a standard presentation using Andrews-Curtis moves.
\end{theorem}

This result was stated and proved by Guo \cite{Guo}, using a theorem of Waldhausen \cite{Waldhausen} on Heegaard splittings of the 3-sphere. However, Guo's proof appears to establish a slightly weaker conclusion. Therefore, in the final section of this paper, we present some further remarks and results, which in addition to Guo's argument, complete the proof of this theorem. Unfortunately, the proof appears to give no information about the number of Andrews-Curtis moves.

\subsection{Outline of the paper.} In Section \ref{Sec:Supplementary}, various modifications to a group presentation are discussed. For balanced presentations of the trivial group, they are shown to be achievable using stable Andrews-Curtis moves. We quantify how many such moves are needed. In Section \ref{Sec:3Sphere}, we introduce triangulations of the 3-sphere and show how a balanced thickenable presentation of the trivial group may be used to produce such a triangulation. In Section \ref{Sec:BoundSAC}, the proof of Theorem \ref{Thm:StableACBound} is completed. Section \ref{Sec:AlternativeApproach} contains an alternative proof of the stable Andrews-Curtis conjecture using classical methods discussed in \cite{HogAngMetzler}. We speculate whether this alternative approach could be used to prove an explicit bound on the number of stable Andrews-Curtis moves.
Section \ref{Sec:ACConjecture} introduces Heegaard splittings and gives a proof of Theorem \ref{Thm:AC}. 

\begin{remark}
In this paper, we work exclusively in the piecewise-linear category. Thus, all cell complexes will be piecewise-linear; in particular, the attaching map of each cell is piecewise-linear. All manifolds will be combinatorial piecewise-linear manifolds. In fact, we will be working only with manifolds of dimension at most 3, and in these dimensions, the topological, piecewise-linear and smooth categories are all equivalent \cite{Moise}.
\end{remark}

\section{Supplementary moves}
\label{Sec:Supplementary}

We now develop some machinery required in the proof of Theorem \ref{Thm:StableACBound}.

Although the stable Andrews-Curtis moves are conjectured to be sufficient to reduce any balanced presentation of the trivial group to the standard one-generator presentation, they do not suffice to relate any two balanced presentations of a non-trivial group. This because the Andrews-Curtis moves do not change the group elements that the generators represent, and the extra stable Andrews-Curtis move only introduces or removes a trivial element. So for example, the presentations $\langle x, y | yx^2, e \rangle$ and $\langle x | e \rangle$ of $\mathbb{Z}$ are not equivalent under the stable Andrews-Curtis moves, since in the former $x = \pm 1$ and $y = \mp 2$ whereas in the latter $x = \pm 1$. Therefore, several authors have considered further modifications that one can make to a presentation without changing its deficiency. (See for example \cite{HogAngMetzler}.) In this section, we consider some of these modifications. In Section \ref{Sec:AlternativeApproach}, we will also recall other modifications that one can make to a presentation.

\subsection{A generalised stable move}
It is quite natural to introduce the following move, which is a generalisation of (4):
\begin{enumerate}
\item[$(4^+)$] given any word $w$ of length at most $2$ in the generators $x_1, \dots, x_n$ and their inverses, introduce a new generator $x_{n+1}$ and a new relation $x_{n+1} w^{-1}$, or the
reverse of this operation, provided the remaining relations do not include the letters $x_{n+1}$ or $x_{n+1}^{-1}$.
\end{enumerate}
Of course, we could have allowed words $w$ of any length, and this move would not have changed the group. However, in that case, there would be infinitely many possible moves that could be applied to a group presentation, which is perhaps not ideal from an algorithmic perspective.

Note that the above two presentations of the infinite cyclic group are related by a $(4^+)$-move. So, for non-trivial groups, moves (0)-($4^+$) are strictly more general than the stable Andrews-Curtis moves. However, the following proposition shows that, for the trivial group, any $(4^+)$-move can be replaced by a sequence of stable Andrews-Curtis moves.

\begin{proposition}
\label{Lem:Move4+}
Let $P$ be a balanced presentation of the trivial group. Suppose that there is a sequence of moves $(0)$-$(4^+)$ with length $m$ that takes $P$ to the standard one-generator presentation. Then there is a sequence of stable Andrews-Curtis moves with length at most $4^{m+3}$ that also takes $P$ to the standard one-generator presentation.
\end{proposition}

A key step in the proof of this is the following lemma, which allows one to express a generator in a presentation of the trivial group as a product of conjugates of the relations and their inverses.

\begin{lemma}
\label{Lem:GensProdsConj}
Let $P = \langle x_1, \dots, x_n | r_1, \dots, r_n \rangle$ be a balanced presentation of the trivial group. Suppose that there is a sequence of moves $(0)$-$(4^+)$ with length $m$ that takes $P$ to the standard one-generator presentation. Then each generator of $P$ is equal in the free group on $x_1, \dots, x_n$ to a product of conjugates of the relations of $P$ and their inverses, where the number of such relations is at most $2^{m+1}-1$ and the length of each conjugating element is at most $2^m$.
\end{lemma}

\begin{proof}
This is by induction on $m$. When $m = 0$, the presentation is $\langle x_1 | x_1 \rangle$, and so when $x_1$ is expressed as a product of conjugates of the relations, only one relation is needed and no conjugations are required. So consider the inductive step, where $P$ is related to $P'$ via a move of type $(0)$-$(4^+)$, and $P'$ is related to the standard one-generator presentation using $m - 1$ such moves. Consider some generator $x_k$ in $P$. Our aim is to write $x_k$ as a product of conjugates of the relations of $P$ and their inverses. 

Suppose first that $x_k$ is introduced in a $(4^+)$ move taking $P'$ to $P$, where the new relation is $x_k w^{-1}$ and $w$ is a word of length at most 2 in the other generators of $P'$. Inductively, $w$ is a product of at most $2(2^m - 1)$ conjugates of relations of $P'$, and where the conjugating words have length at most $2^{m-1}$. So $x_k = (x_k w^{-1}) w$ is a product of at most $2(2^m - 1) + 1 = 2^{m+1} - 1$ conjugates of relations of $P$, and where the conjugating words again have length at most $2^{m-1}$. This proves the inductive step in this case. 

Now suppose that the move from $P'$ to $P$ is a type $(4^+)$ move that introduces a generator $x_n$ other than $x_k$.
Then $x_k$ is also a generator in $P'$ and hence inductively can be written as a word $u$ which is a product of conjugates of at most $2^{m} - 1$ relations in $P'$ and their inverses, where the length of each conjugating element is at most $2^{m - 1}$. This word $u$ is the required product of conjugates of the relations of $P$ and their inverses.

Consider the case where the move from $P'$ to $P$ is a type $(4^+)$ move that removes a generator $x_{n+1}$ and a relation $x_{n+1} w^{-1}$ where $w$ is a word of
length at most $2$ in the remaining generators. Inductively, $x_k$ is a word $u$ which is a product of conjugates of at most $2^{m} - 1$ relations in $P'$ and their inverses, where the length of each conjugating element is at most $2^{m - 1}$. There is a homomorphism $\phi$ from the free group $\langle x_1, \dots, x_{n+1} \rangle$ to the free group $\langle x_1, \dots, x_n \rangle$ sending $x_{n+1}$ to $w$ and sending $x_i$ to $x_i$ for each $i \leq n$. Apply this homomorphism $\phi$ to the word $u$. This sends the relation $x_{n+1}w^{-1}$ to the identity element, and it sends the remaining relations of $P'$ to themselves. The effect on each conjugating word is to increase its length by at most a factor of $2$. Thus, $x_k$ is equal in $\langle x_1, \dots, x_n \rangle$ to $\phi(u)$, which is a product of at most $2^{m} - 1$ relations in $P$ and their inverses, and where the each conjugating word has length at most $2^m$. This proves the induction in this case.

We now consider the remaining cases where $P$ is obtained from $P'$ by an Andrews-Curtis move of type (0)-(3).
If the move from $P'$ to $P$ replaces $r_i$ by $r_i^{-1}$ or introduces or removes a generator and its inverse to some relation, then again the same word $u$ can be used for $x_k$. If the move from $P'$ to $P$ replaces some $r_i$ by $r_i r_j$, then we replace each instance of $g r_i g^{-1}$ in $u$, where $g$ is a conjugating element, by $g (r_i r_j) g^{-1} g r_j^{-1} g^{-1}$. This increases the number of relations in the word by at most a factor of $2$ and it does not increase the length of any conjugating word. The same argument works when the Andrews-Curtis move replaces some $r_i$ by $r_j r_i$ for some $j$. Finally, when the Andrews-Curtis move replaces some $r_i$ by $x_j r_i x_j^{-1}$ or $x_j^{-1} r_i x_j$, then the number of relations in the word remains unchanged but the conjugating elements increase their length by at most 1.
\end{proof}

\begin{lemma}
\label{Lem:RelationLength}
Let $P = \langle x_1, \dots, x_n | r_1, \dots, r_n \rangle$ be a balanced presentation of the trivial group. Suppose that there is a sequence of moves $(0)$-$(4^+)$ with length $m$ that takes $P$ to the standard one-generator presentation. Then each $r_i$ has length at most $2^{m+1}$.
\end{lemma}

\begin{proof}
This is by induction on $m$. For $m = 0$, the only relation has length $1$ and so the lemma holds in this case. For the inductive step, suppose that $m \geq 1$ and that the claim is true for sequences of moves with length at most $m-1$.  Each move of type (0) or (3) increases the maximal length of a relation by at most $2$. Moves of type (1) do not change the lengths of the relations. Each move of type (2) at most doubles the maximal length of each relation. Each move of type $(4^+)$ introduces a relation of length at most 3. Hence, the maximal relation length is at most $\max \{  2^m + 2, 2^{m+1}, 3 \} = 2^{m+1}$, as required.
\end{proof}

\begin{proof}[Proof of Proposition \ref{Lem:Move4+}]
Suppose that $P = \langle x_1, \dots, x_n | r_1, \dots, r_n \rangle$ can be converted into $P'$ using a $(4^+)$ move, and that $P'$ is related to the standard one-generator presentation using at most $m-1$ moves of the form $(0)$-$(4^+)$. 

Say that the $(4^+)$ move taking $P$ to $P'$ introduces the generator $x_{n+1}$ and the relation $x_{n+1} w^{-1}$, where $w$ is a word of length at most $2$ in the remaining generators. By Lemma \ref{Lem:GensProdsConj}, each generator is equal, in the free group on $x_1, \dots, x_n$, to a product of conjugates of the relations of $P$ and their inverses, where the number of such relations is at most $2^{m+1}-1$ and the length of each conjugating element is at most $2^{m}$. So $w$ is equal, in $\langle x_1, \dots, x_n \rangle$, to a product of conjugates of the relations of $P$ and their inverses, where the number of such relations is at most $2^{m+2}-2$ and the length of each conjugating element is at most $2^{m}$. This product reduces to $w$ by a sequence of type (0) moves. The number of such moves is at most half the length of the initial expression plus 1. Hence, by Lemma \ref{Lem:RelationLength}, the number of moves is at most $(2^{m+2}-2)(2^{m+1} + 2^{m+1}) + 1 = 2^{2m+4} -2^{m+3} + 1$.

Consider, for example, where only one relation is used in this expression for $w$ and so $w$ is equal in the free group $\langle x_1, \dots, x_n \rangle$ to $g r_i g^{-1}$ for some $i$ and for some word $g$ with length at most $2^{m}$. Then, we relate $P$ to $P'$ using the following sequence of stable AC moves. First, introduce the generator $x_{n+1}$ and relation $x_{n+1}$. Then replace $r_i$ by $g^{-1} r_i^{-1} g$ using a type (1) move and a sequence of (3) moves with length equal to the length of $g$. Then multiply the relation $x_{n+1}$ by this new relation $g^{-1} r_i^{-1} g$. Then apply (0) moves to reduce this relation to $x_{n+1} w^{-1}$. Finally, replace the relation $g^{-1} r_i^{-1} g$ by $r_i$ by undoing the conjugations and applying type (0) moves followed by a type (1) move. In general the number of stable AC moves that we need is at most $1 + (2^{m+2}+3)(2^{m+2}-2) + 2^{2m+4} - 2^{m+3} +1 < 2^{2m+5}$, where the initial 1 arises from the initial $(4)$ move, and the factor $2^{m+2}-2$ is the upper bound for the number of relations. 

The same argument works if the $(4^+)$ move taking $P$ to $P'$ removes the generator $x_{n+1}$ and the relation $x_{n+1} w^{-1}$. We simply reverse the above sequence of moves, replacing the $(4^+)$ move by at most $2^{2m+5}$ stable AC moves.

Now applying this to each of the $(4^+)$ moves in the sequence, the number of stable AC moves is most
$\sum_{k=1}^m 2^{2k+5} \leq 4^{m+3}$.
\end{proof}

\subsection{Nielsen equivalent generators}
\label{Sec:Nielsen}
As discussed above, the stable Andrews-Curtis moves do not change the group elements that the generators represent. Aside from the $(4^+)$ move, one can also deal with this by changing the generators via an automorphism of $\langle x_1, \dots, x_n \rangle$. Specifically, suppose that $\phi \colon \langle x_1, x_1^{-1}, \dots, x_n, x_n^{-1} \rangle \rightarrow \langle x_1, x_1^{-1} \dots, x_n, x_n^{-1} \rangle$ is a semi-group homomorphism that descends to a group automorphism of $\langle x_1, \dots, x_n \rangle$. Then from the presentation $\langle x_1, \dots, x_n | r_1, \dots, r_n \rangle$, one obtains a new presentation $\langle x_1, \dots, x_n | \phi(r_1), \dots, \phi(r_n) \rangle$. Unless $\phi$ fixes  a generator $x_i$, it is convenient to relabel $x_i$ in the latter presentation by $x'_i$, since it will in general represent a new element of the group.

Recall \cite{Nielsen} that the automorphism group of the free group is finitely generated by the Nielsen generators. When applied to a group presentation, some of these do not make an essential change to the presentation. For example, when an automorphism of $\langle x_1, \dots, x_n \rangle$ that cyclically permutes the generators is applied, the effect on a finite presentation $\langle x_1, \dots, x_n | r_1, \dots, r_n \rangle$ can be viewed simply as relabelling of the generators. However, some Nielsen generators do have a substantive effect on a presentation. These are as follows:
\begin{enumerate}
\item[(5)] Replace generator $x_i$ by $x'_i$, and replace each instance of $x_i$ in any relation by $x'_i x_j$ for some $j\not=i$, and each instance of 
$x_i^{-1}$ by $x_j^{-1} (x_i')^{-1}$.
\item[(6)] Replace generator $x_i$ by $x'_i$, and replace each instance of $x_i$ in any relation by $(x'_i)^{-1}$ and each instance of $x_i^{-1}$ by $x'_i$.
\end{enumerate}

\begin{lemma}
\label{Lem:Nielsen}
Each type $(5)$ or $(6)$ move can be replaced instead by a composition of at most $10\ell+3$ instances of moves $(0)$-$(4^+)$, where $\ell$ is length of the presentation.
\end{lemma}

\begin{proof}
We show how move (5) can be achieved using moves (0)-($4^+$). First introduce a new generator $x'_i$ as well as a relation $x'_i x_j x_i^{-1}$.
Now consider any relation of the presentation. We will cyclically permute the letters in this relation by applying a sequence of type (3) and (0) moves.
As we do this, instances of $x_i$ come to the front of the relation. At that point, we can premultiply by the new relation and then cancel $x_i^{-1} x_i$, to replace $x_i$ by 
$x'_i x_j$. Similarly, if an instance of $x_i^{-1}$ comes to the front, we first replace the relation $x'_i x_j x_i^{-1}$ by its inverse $x_i x_j^{-1} (x_i')^{-1}$, we then
cyclically permute it to get $x_j^{-1} (x_i')^{-1} x_i$, we then premultiply our given relation by this and cancel $x_i x_i^{-1}$, and then we return the extra relation to $x'_i x_j x_i^{-1}$.
Finally remove the generator $x_i$ and the relation $x'_i x_j x_i^{-1}$ using $(1)$ and $(4^+)$. The number of moves we have performed is the number of $(4^+)$ moves and the final $(1)$ move, which is 3, 
plus the total number of cyclic permutations, which is at most $\ell$, multiplied by the number of moves involved at each instance of $x_i$ or $x_i^{-1}$, which is at most $10$. 

Essentially the same argument works for move (6).
\end{proof}

\subsection{Removing the supplementary moves}
\label{Subsec:RemoveMoves}

\begin{proposition} 
\label{prop:Remove4Plus56}
Let $P$ be a balanced presentation of the trivial group. Suppose that $P$ can be transformed to the
standard one-generator presentation using a sequence of at most $m$ moves of type $(0)$, $(1)$, $(2)$, $(3)$, $(4^+)$, $(5)$ and $(6)$,
and that whenever such a move is performed, the length of the presentation is at most $L$.
Then $P$ can also be transformed to the standard one-generator presentation using at most 
$4^{12Lm + 3}$
stable AC moves.
\end{proposition}

\begin{proof}
We first remove each instance of move (5) or (6) using Lemma \ref{Lem:Nielsen}. By assumption,
the length of the presentation when such a move is performed is at most $L$. We may assume
that $L$ is at least 2, as otherwise there is no need to apply the move. Hence, the upper bound 
$10L+3$ from Lemma \ref{Lem:Nielsen} is at most $12L$. So, the result
is a sequence of moves $(0)$ to ($4^+$) with length at most $12Lm$.  Then Proposition \ref{Lem:Move4+}
replaces this with a sequence of stable AC moves with length at most $4^{12Lm+3}$,
as required.
\end{proof}

Our main result, Theorem \ref{Thm:StableACBound}, follows immediately from this proposition as well as the following theorem.

\begin{theorem}
\label{Thm:ThickenableBound}
Let $P$ be a thickenable balanced presentation of the trivial group with length $\ell$. Then $P$ can be transformed 
to the standard one-generator presentation using a sequence of at most
$2^{k\ell^2}$
moves of type $(0)$, $(1)$, $(2)$, $(3)$, $(4^+)$, $(5)$ and $(6)$, where $k \leq 7 \cdot 10^5$. Moreover, when each such move
is applied, the length of the presentation is at most
$2^{k\ell^2}$.
\end{theorem}

Note that the upper bound on the number of moves here is a single exponential (of $k\ell^2$), whereas the upper
bound in Theorem \ref{Thm:StableACBound} is a double exponential. The increase in the height of this exponential tower
is solely due to Proposition \ref{prop:Remove4Plus56}. In summary, the reason for this double exponential rather than a single exponential is because
of the restrictive nature of the stable Andrews-Curtis moves.

\subsection{Triangular presentations}

We give one sample application of the $(4^+)$ move.

\begin{definition}
A presentation of a group is \emph{triangular} if each relation has length 3.
\end{definition}

\begin{lemma} 
\label{Lem:MakeTriangular}
Let $P$ be a balanced group presentation with length $\ell$ and with $n$ generators. Then there is a sequence of at most $4\ell + 18n$ moves of type $(0)$-$(4^+)$ taking
$P$ to a triangular presentation $P'$.
\end{lemma}

\begin{proof}
We first make each relation have length at least three, by multiplying it by $x_1^2 x_1^{-2}$ if necessary. The number of moves we have made so far is at most $2n$, and the length of the resulting presentation is at most $\ell + 4n$. We now make each relation have length exactly three. Suppose that there is a relation $r_i$ with length more than three, and let $ab$ be its first two letters. We add a new generator $x_{j}$ and a new relation $x_{j} b^{-1} a^{-1}$. We then premultiply $r_i$ by the relation and apply (0) twice to remove $b^{-1} a^{-1} a b$. This removes the initial $ab$ and replaces it by $x_j$. Repeating this at most $\ell + 4 n$ times, we obtain the required triangular presentation.
\end{proof}

\subsection{2-complexes with more than one 0-cell}
\label{Subsec:2Complexes}
It will be convenient to work with finite connected 2-complexes $K$ which may have more than one 0-cell. These specify a presentation for $\pi_1(K)$ once a choice of maximal tree in the 1-skeleton has been made, a choice of orientation is made for each 1-cell, and a basepoint and orientation for the boundary of each 2-cell has been picked. The maximal tree can be collapsed to a single 0-cell. The remaining 1-cells give generators for the presentation, and each 2-cell with a basepoint and an orientation specifies a relation. We say that a choice of orientation on each edge and a choice of orientation and basepoint for the boundary of each 2-cell is a \emph{choice of orientations and basepoints} for the 2-complex.

The choice of orientations and basepoints has only a modest effect on the presentation. A change of orientation on some 1-cell in the maximal tree has no effect on the presentation. A change of orientation on a 1-cell not in the tree replaces the resulting generator by its inverse. In other words, the effect on the presentation is Nielsen move (6). A change of basepoint for some relation results in a cyclic permutation of the relation, which can be achieved by conjugation moves (3) and cancellation moves (0), and a change of orientation replaces the relation by its inverse, which is move (1). In this subsection, we explore the effect of changing the maximal tree, via the following move.

Note that here and throughout this paper, we will use the term \emph{graph} to mean a 1-complex. So, edge loops are permitted, as are multiple edges between a pair of vertices.

\begin{definition} 
Let $G$ be a finite connected graph, and let $X$ be a maximal tree in $G$. Let $y$ be some edge not in the tree. 
Let $p$ be the unique embedded path in $X$ (up to reparametrisation and inversion) between the endpoints of $y$. Let $x$ be some edge in $p$. Removing $x$ from $X$ and replacing it by $y$
gives a new maximal tree. We term this operation as \emph{performing an edge swap}.
\end{definition}

\begin{lemma}
\label{Lem:EdgeSwap}
Let $K$ be a finite connected 2-complex with a choice of orientations and basepoints, and let $X$ be a maximal tree in its 1-skeleton.
Let $P$ be the finite presentation obtained from $K$ and $X$. Let $n$ be the number of generators of $P$. Let $X'$ be obtained from $X$ by swapping the edge $x$ in $X$ with the edge $y$ not in $X$. Let $P'$ be the new presentation for $\pi_1(K)$. Say that the lengths of the presentations $P$ and $P'$ are $\ell$ and $\ell'$. Then $P'$ is obtained from $P$ by at most $4(n-1)$ moves of Nielsen type $(5)$ and $(6)$ and at most $2\ell + \ell'$ moves of type $(0)$.
\end{lemma}

\begin{proof}
Let $\overline{K}$ be the 2-complex obtained from $K$ by collapsing all the edges in $X - x$. Thus, $X$ is collapsed to the single edge $x$. The union of $x$ and $y$ then becomes a cycle of length 2. We obtain $P$ by collapsing $x$ in this cycle and forming the resulting presentation, and we obtain $P'$ by instead collapsing $y$. Each 2-cell in $\overline{K}$ descends to a relation in $P$ by removing each instance of $x$ and $x^{-1}$, and a relation in $P'$ by removing each instance of $y$ and $y^{-1}$. Say that we are given such a relation $r$ in $P$. We wish to construct the corresponding relation $r'$ in $P'$. 

Suppose that in the 2-cycle in $\overline{K}$ formed by the edges labelled $x$ and $y$, the edges appear in the order $xy$ rather than $xy^{-1}$. The argument in the latter case is similar and is omitted.

Let $w$ be the attaching word of a 2-cell of $\overline{K}$. We modify this using a sequence of type (0) expansions as follows. Between each pair of successive letters in $w$, including the final and initial letter, examine whether that point in the 2-cell lies at the start or the end of $x$. At each such point at the end of $x$, insert $x^{-1}x$ into the word $w$. This only changes the resulting relations $r$ and $r'$ by type (0) moves. But we can view this procedure as modifying each of the letters in $w$, as follows. Replace each letter $a$ corresponding to an edge in $\overline{K}$ with both start and endpoint equal to the endpoint of $x$ by $x a x^{-1}$. Replace each letter $b$ corresponding to an edge in $\overline{K}$ starting at the start of $x$ and ending at the end of $x$ by $bx^{-1}$. Replace each letter $c$ starting at the endpoint of $x$ and ending at the starting point of $x$ by $xc$. In particular replace $y$ by $xy$ and replace $x$ by $xx^{-1}$. Make the corresponding replacements for $a^{-1}$, $b^{-1}$, and $c^{-1}$. The advantage of this approach is that, given only $r$, we can form $r'$ up to type (0) moves, by replacing $y$ by $x$ and making the above modifications to the remaining letters of the word. The crucial observation is that these modifications are achieved by Nielsen moves. We need at most two type (5) Nielsen moves and at most two type (6) moves for each generator other than $y$. So at most $4(n-1)$ such moves are needed in total.

An example may be helpful. Suppose that the 1-skeleton of $\overline{K}$ has 4 edges, where $x$ and $y$ form a cycle of length $2$, and with $a$ forming a loop attached to the end of $x$, and $d$ forming a loop attached at the start of $x$. Say that a 2-cell in $\overline{K}$ is attached along the word $w = xaayd$. Then $r$ becomes $aayd$ and $r'$ is $xaad$. Performing the above procedure to $r$ gives $(x a x^{-1}) (x a x^{-1}) x d$, which freely reduces to $r'$. This is achieved by the automorphism
$$y \mapsto x, \qquad a \mapsto xax^{-1}, \qquad d \mapsto d.$$
This is a composition of two type (5) Nielsen moves and two type (6) moves.

The final step that we make is to apply type (0) moves to reach the relations in $P'$. We freely reduce each relation until it freely reduced. This requires at most $2\ell$ type (0) moves (where the factor of $2$ arises from the fact that in the above procedure we may initially increase the length of the relations of $P$). Then we expand to reach the required relations in $P'$. This requires at most $\ell'$ moves.
\end{proof}

\begin{lemma}
\label{Lem:EdgeSwapForest}
Let $G$ be a finite connected graph, and let $X$ be a maximal tree. Let $Y$ be a subgraph of $G$ that is a forest consisting of at most $k$ edges. Then there is a sequence of at most $k$ edge swaps taking $X$ to a maximal tree that contains $Y$.
\end{lemma}

\begin{proof}
We prove this by induction on the number of edges in $Y$ but not $X$. When this number is zero, then $Y$ is a subset of $X$ and no edge swaps are required. So suppose that there is an edge $y$ in $Y$ but not $X$. Note that the endpoints of $y$ are distinct because $y$ is part of a forest.
Let $p$ be the path in $X$ joining the endpoints of $y$. Then $p \cup y$ is a cycle and so does not lie entirely in $Y$, since $Y$ is a forest. So there is an edge $x$ in $p$ lying in $X$ but not $Y$. We may perform an edge swap, removing $x$ from $X$ and adding $y$. The induction step is therefore established.
\end{proof}

\section{Triangulations of the 3-sphere}
\label{Sec:3Sphere}

A central part of our argument will be the use of triangulations of the 3-sphere. Throughout, we will use the notion of a triangulation that has become standard in low-dimensional topology. For a 3-manifold, this is an expression of the manifold as a collection of tetrahedra with their faces identified in pairs via affine homeomorphisms. Hence, it is possible for a triangulation of a 3-manifold to have a single vertex for example, whereas this is not possible when triangulations are required to arise as the topological realisation of a simplicial complex.

\subsection{Thickenable group presentations}

By definition, a group presentation $P$ is thickenable if its associated 2-complex $K$ embeds in some 3-manifold. Its regular neighbourhood $N(K)$ admits a natural handle structure,
where each $i$-cell of $K$ thickens to an $i$-handle. This regular neighbourhood $N(K)$ deformation retracts onto $K$. Hence, when $K$ is the 2-complex
arising from a balanced presentation of the trivial group, $N(K)$ is a compact contractible 3-manifold. By Perelman's solution to the 3-dimensional Poincar\'e conjecture \cite{Perelman1, Perelman2, Perelman3}, $N(K)$ is therefore a 3-ball.

\subsection{Thickenable triangular presentations}

We now give a version of Lemma \ref{Lem:MakeTriangular} for thickenable group presentations. For simplicity, we assume that the initial presentation is balanced and represents the trivial group.

\begin{lemma} 
\label{Lem:MakeTriangularThickenable}
Let $P$ be a balanced thickenable presentation of the trivial group with $n$ generators and length $\ell$. Then there is a sequence of at most $5\ell (n+1)$ moves of type 
$(0)$-$(4^+)$ taking
$P$ to a thickenable triangular presentation with at most $\ell$ generators and $\ell$ relations or to the standard one-generator presentation.
\end{lemma}

\begin{proof}
Let $K$ be the 2-complex corresponding to $P$. Let $N(K)$ be its 3-manifold regular neighbourhood, which admits a natural handle structure.

Note first that no relation of $P$ has length zero, since we could remove this relation and obtain a presentation with more generators than relations, which would necessarily form an infinite group.

Suppose that some relation has length $1$. Then this corresponds to a 2-handle in $N(K)$ that runs over a single 1-handle. We can collapse these two handles. Say that $x_i$ is the generator corresponding to the collapsed 1-handle. Then this has the effect of removing each occurrence of $x_i$ or its inverse in any other relation. This is achieved by at most $5\ell$ stable AC moves: gradually cyclically permute each relation by using a sequence of conjugations and cancellations, and when an $x_i$ or its inverse appears at the front, remove it using moves (3) and (0), possibly preceded by move (1). Then remove the generator and relation $x_i$ or $x_i^{-1}$. We have used at most $5\ell$ stable AC moves, and we have removed a generator.

Suppose now that some relation has length $2$. This corresponds to a 2-handle that runs over two 1-handles, or possibly the same 1-handle twice. However, it cannot in fact run over the same 1-handle twice, since the relation would then spell the word $x_i^2$, $x_i^{-2}$, $x_i x_i^{-1}$ or $x_i^{-1} x_i$, but none of these is a primitive element of the free abelian group generated by $x_1, \dots, x_n$, whereas each relation must be primitive in this group. Thus, we deduce that the relation spells the word $x_i^{\pm 1} x_j^{\pm 1}$, for $j \not = i$. As above, we can replace each occurrence of $x_i$ or $x_i^{-1}$ in the relations by $x_j$ or $x_{j}^{-1}$, and we can then remove this relation and $x_i$ by move ($4^+$) possibly preceded by (1), (3) and (0). At the level of the thickened 2-complex, this is achieved by crushing the 2-handle and the two incident 1-handles to a single 1-handle. In particular, the presentation remains thickenable. Again this uses at most $5\ell$ moves of type $(0)$-$(4^+)$, and again we have removed a generator.

Once the above modifications have been completed, we have used at most $5\ell n$ moves, and the presentation length is still at most $\ell$. But now each relation has length at least $3$.

We now perform the moves described in the proof of Lemma \ref{Lem:MakeTriangular}. These subdivide each 2-cell with length more than 3 into triangles. This subdivision does not change the topological type of the resulting 2-complex. Hence, it remains thickenable. This uses at most $4(\ell-3)$ moves. In total, the number of moves is at most $5 \ell n + 4(\ell-3) \leq 5 \ell (n+1)$.
\end{proof}

\subsection{A 3-sphere triangulation}
\label{Subsec:3SphereTri}
We now describe a procedure that takes a balanced thickenable triangular presentation $P$ of the trivial group and creates a triangulation of the 3-sphere. Let $K$ be the presentation 2-complex for $P$. Since $K$ is thickenable, it lies inside a regular neighbourhood $N(K)$, which is a 3-ball with a natural handle structure. The boundary of $N(K)$ is a 2-sphere. Attach onto this 2-sphere a 3-ball $B$, via a homeomorphism from $\partial B$ to $\partial N(K)$. Now collapse $N(K)$ back to $K$, forming a cell structure for the 3-sphere. The boundary of the 3-cell $B$ is triangulated. We will show below that a vertex $v$ in this 2-sphere can be found so that:
\begin{enumerate} 
\item no edges in $\partial B$ start and end at $v$;
\item when the star of $v$ is removed from $\partial B$, the result is a triangulation of a closed disc $D$. 
\end{enumerate}
Then $B$ can be triangulated by forming the cone over $D$ with cone point $v$. (See Figure \ref{Fig:ConeOverDisc}.) This provides the required 3-sphere triangulation.

\begin{figure}[h]
\centering
\includegraphics[width=0.5\textwidth]{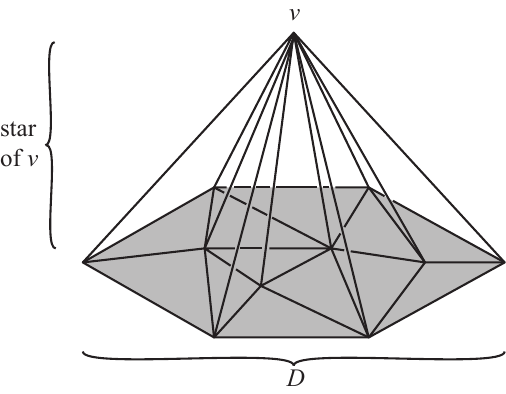}
\caption{The triangulation of the ball $B$}
\label{Fig:ConeOverDisc}
\end{figure}

We note that, when (0) holds, the only obstruction to (2) above is the existence of a cycle of length $2$ in the 1-skeleton of $\partial B$ containing $v$. Thus, a suitable vertex $v$ can be chosen using the following lemma.

\begin{lemma}
In any triangulation of the 2-sphere, there is a vertex that misses all cycles in the 1-skeleton with length $1$ or $2$.
\end{lemma}

\begin{proof}
A cycle of length 1 is a simple closed curve, and any two such simple closed curves are either disjoint or intersect at their common vertex. Hence, the union of the cycles of length 1 is a collection of wedges of circles. These divide the 2-sphere into planar surfaces, at least one of which must be a disc. We focus on this disc $F$. Then the interior of the disc contains no cycles of length 1. When there are no cycles of length 1, we let $F$ denote the original 2-sphere.

The cycles of length 2 in the 2-sphere cannot intersect both the interior of $F$ and the complement of $F$. Hence, we will focus on the cycles of length 2 within $F$. Our aim is to find a vertex $v$ in the interior $F$ that avoids these cycles. Note that if $F$ contains no cycle of length $2$, then we may let $v$ be any vertex in the interior of $F$, of which there must be at least one. So suppose now that $F$ does contain a cycle of length $2$. Consider a cycle of length $2$ in $F$ that is innermost, in the sense that it bounds a disc in $F$ with the property that there is no other cycle of length $2$ in the disc. The boundary of the disc has two vertices $v_1$ and $v_2$ and two edges. Other than these two edges, there are no other edges joining $v_1$ and $v_2$ in the disc. So, no cycle of length 2 in $F$ intersects the interior of this disc. Now let $v$ be any vertex in the interior of this disc.
\end{proof}

\subsection{Pachner moves}

Any two triangulations of the same closed piecewise-linear manifold differ by a sequence of \emph{Pachner moves} \cite{Pachner}. In the case of triangulations of a 3-manifold, these are called 1-4 and 2-3 moves, and their inverses. These are shown in Figure \ref{Fig:PachnerMoves}. In a 1-4 move, the interior of a tetrahedron $\Delta$ is removed and the triangulated cone over $\partial \Delta$ is inserted. A 4-1 move is the reverse of this procedure. In a 2-3 move, two distinct tetrahedra that share a common face are considered. The union of the interiors of these tetrahedra and the interior of the face is an open 3-ball. The contents of this open 3-ball is removed, and replaced by three new tetrahedra arranged around a common edge.
A 3-2 move reverses this.

\begin{figure}[h]
\centering
\includegraphics[width=0.9\textwidth]{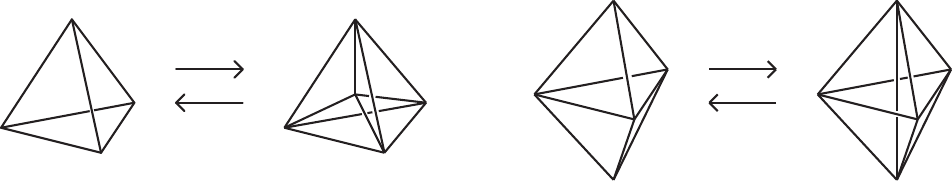}
\caption{The Pachner moves for closed 3-manifolds}
\label{Fig:PachnerMoves}
\end{figure}

The following key result of Mijatovic \cite{Mijatovic} will be central to our arguments. See also King's paper \cite{King}, which contains a similar
result using `bistellar moves' rather than Pachner moves.

\begin{theorem}
\label{Thm:PachnerS3}
Any triangulation of the 3-sphere with $t$ tetrahedra may be converted to the standard triangulation using at most
$$ a t^2 2^{bt^2}$$
Pachner moves, where $a \leq 6 \cdot 10^6$ and $b \leq 5 \cdot 10^4$.
\end{theorem}

Here, the \emph{standard triangulation} of the 3-sphere is obtained from two tetrahedra by identifying their boundaries homeomorphically.

\subsection{A choice of two spanning trees}
\label{Subsec:TwoTrees}
While Theorem \ref{Thm:PachnerS3} is a striking result, it is not immediately clear how it can be used to create a sequence of stable Andrews-Curtis moves. This is because stable Andrews-Curtis moves do not change the Euler characteristic of the relevant 2-complexes, whereas Pachner moves do change the Euler characteristic of a triangulation's 2-skeleton. To deal with this, we introduce auxiliary data, as follows.

\begin{definition}
Given a triangulation of the 3-sphere, its \emph{dual 1-skeleton} is a graph with a vertex corresponding to each tetrahedron and with an edge corresponding to each face, where the edge joins the vertices dual to the two tetrahedra incident to the face.
\end{definition}

\begin{definition}
A \emph{3-sphere triangulation with auxiliary trees} is a triple $(T, X, X')$, where $T$ is a triangulation of the 3-sphere, $X$ is a maximal tree in its 1-skeleton and $X'$ is a maximal tree in its dual 1-skeleton. Furthermore, a choice of orientations and basepoints is made for the 2-skeleton of $T$. However, for simplicity of notation, we suppress this orientation and basepoint information from our terminology. The \emph{associated 2-complex} $K(T,X,X')$ is the 2-complex obtained by removing the interior of all tetrahedra of $T$, the interior of every face dual to an edge in $X'$, and then collapsing $X$ to a single 0-cell. The orientations and basepoints on the 2-skeleton descend to orientations and basepoints for $K(T,X,X')$. The \emph{associated group presentation} $P(T,X,X')$ is the presentation associated with $K(T,X,X')$.
\end{definition}

Note that the interiors of the 3-simplices of $T$ and the interiors of the 2-simplices dual to the edges of $X'$ have union equal to an open 3-ball. Hence, $K(T,X,X')$ is the 2-skeleton for a cell structure of the 3-sphere with a single 3-cell. In particular, its fundamental group is trivial, and it has Euler characteristic is 1. Thus $P(T,X,X')$ is a balanced presentation of the trivial group.

\subsection{Edge swaps on the dual tree}
In Section \ref{Subsec:2Complexes}, we considered how to change a spanning tree in the 1-skeleton of a 2-complex, by means of an edge swap. In Lemma \ref{Lem:EdgeSwap}, we saw the effect of edge swaps on the associated group presentations. In this subsection, we analyse the effect of an edge swap on a dual tree.

\begin{lemma}
\label{Lem:EdgeSwapDualTree}
Let $(T,X,X')$ be a 3-sphere triangulation with auxiliary trees, where $t$ is the number of tetrahedra. Suppose $X'_2$ is a maximal tree in the dual 1-skeleton that is obtained from $X'$ by an edge swap. Then $P(T,X,X')$ and $P(T,X,X'_2)$ differ by a sequence of at most $4t^2 + 12t - 16$ Andrews-Curtis moves.
\end{lemma}

\begin{proof}
Let $e$ be the edge in $X'$ that is removed in the swap, and let $e_2$ be the added edge. Dual to $e$ and $e_2$ are faces $f$ and $f_2$ of the triangulation $T$. The tree $X'$ specifies a collection of tetrahedra with some of their faces glued in pairs. Let $B$ be the resulting triangulated 3-manifold. This deformation retracts onto $X'$, and hence is a 3-ball. The boundary of $B$ has two copies of $f_2$ in its boundary. Furthermore, $f$ is properly embedded in $B$. Cutting $B$ along $f$ gives two 3-balls $B_1$ and $B_2$. Each contains a copy of $f_2$, because $e$ lies in the path in $X'$ joining the endpoints of $e_2$. The total number of triangles in $\partial B$ is at most $4t-2$, since each face of $T$, other than $f$, appears at most twice in it. So, one of $B_1$ or $B_2$ has at most $2t$ triangles in its boundary, say $B_1$. When $f$ and $f_2$ are removed from $\partial B_1$, the result is an annulus (in the case where $\partial f$ and $\partial f_2$ are disjoint in $\partial B_1$) or obtained from an annulus by collapsing some properly embedded vertical arcs. So, we may freely homotope $\partial f$ across this (possibly collapsed) annulus, taking it to $\partial f_2$. This homotopy may be achieved in a series of steps, where at each step, we homotope across a single triangle. (See Figure \ref{Fig:SlideRelationBall}.) We view $f$ as a 2-cell attached onto the remainder of $K(T,X,X')$. This homotopy represents a modification to the attaching map of this 2-cell. After the homotopy has been completed, the resulting 2-complex is  $K(T,X,X'_2)$.

We now analyse the effect on the attaching word of the 2-cell from each step of the homotopy. At each step, we slide an edge of the 2-cell across a triangle. By applying the conjugation moves (3) and cancellation moves (0), we may change the basepoint of the 2-cell, so that the first edge it runs over is the relevant edge. The triangle specifies a relation of the presentation, and we can premultiply by a conjugate of this relation or its inverse and perform at most two cancellations to achieve the slide. We may first need to cyclically permute the relation by a generator or its inverse so that the edge that we slide across is the final letter of the relation. Each step increases the length of the attaching word by at most $1$. So after $2t-2$ such steps, the attaching word has length at most $2t+1$. Therefore, the number of Andrews-Curtis moves used in each step is at most $2t+8$ (where at most $2(t+1)$ of these arise from a cyclic permutation of the attaching word of the 2-cell). So, the total number of Andrews-Curtis moves is at most $(2t - 2)(2t+8) = 4t^2 + 12t - 16$.
\end{proof}

\begin{figure}[h]
\centering
\includegraphics[width=0.5\textwidth]{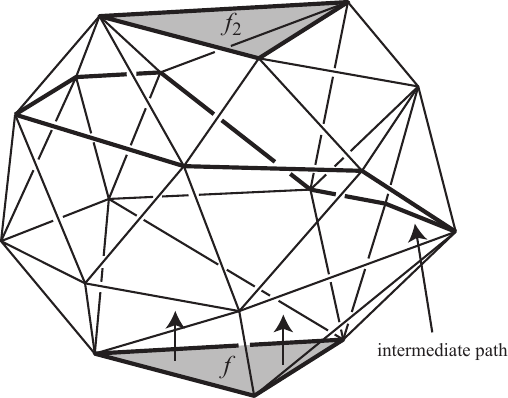}
\caption{The 3-ball $B_1$ is shown with $f$ and $f_2$ in its boundary. A homotopy is performed, sliding the 2-cell across the annulus $\partial B_1 \cut (f \cup f_2)$. A potential path in the middle of this homotopy is shown.}
\label{Fig:SlideRelationBall}
\end{figure}

\section{The bound on stable Andrews-Curtis moves}
\label{Sec:BoundSAC}
In this section, we prove Theorem \ref{Thm:ThickenableBound}. As explained in Section \ref{Subsec:RemoveMoves}, this result, combined with Proposition \ref{prop:Remove4Plus56}, immediately implies Theorem \ref{Thm:StableACBound}.

We are given a balanced thickenable presentation $P$ of the trivial group with length $\ell$. We may assume that its number of generators is at most $\ell-1$, as otherwise there is a sequence of at most $2(\ell-1)$ moves of type (4) and possibly $(1)$ taking it to the standard one-generator presentation. This presentation $P$ specifies a 2-complex $K$ with a single 0-cell. Furthermore, each 1-cell of $K$ is oriented, and the boundary of each 2-cell is based and oriented.

Using Lemma \ref{Lem:MakeTriangularThickenable}, we can convert this to a balanced thickenable triangular presentation $P'$ or to the standard one-generator presentation. This uses at most $5\ell^2$ moves (0)-($4^+$) and the result $P'$ has at most $\ell$ relations. This corresponds to a triangulated 2-complex $K'$ with a single 0-cell. We can attach on a triangulated 3-ball $B$ to $K'$ to form a 3-sphere triangulation $T$, as in Section \ref{Subsec:3SphereTri}. The triangulation of $B$ is obtained by picking a vertex $v$ in $\partial B$ that does not lie on a cycle of length 1 or 2 in the 1-skeleton of $\partial B$, letting $D$ be the complement of the star of $v$ in $\partial B$, and then forming the cone over $D$ with cone point $v$.  The number of tetrahedra is at most twice the number of relations of $P'$, hence at most $2\ell$.

We now bound the number of vertices in the interior of this triangulated disc $D$. Let $V$, $E$ and $F$ denote the number of vertices, edges and faces of the sphere $\partial B$. Then $V - E + F = 2$. Also, $3F = 2E$, and so $V = (F/2) + 2$. The number of faces $F$ is twice the number of relations, so at most $2\ell$. Hence, $V$ is at most $\ell + 2$. The vertex $v$ does not lie in $D$, and also there is at least one vertex in $\partial D$. So the number of vertices in the interior of $D$ is at most $\ell$.

The cell complex $K'$ associated with $P'$ has a single 0-cell. Furthermore, no new 0-cells are introduced when $B$ is attached. Hence, $T$ also has a single 0-cell. So its maximal tree $X$ is just this 0-cell. The 3-ball $B$ has a dual 1-complex. Let $X'$ be some maximal tree in this dual 1-complex. Thus, $(T, X, X')$ forms a 3-sphere triangulation with auxiliary trees. As described in Section \ref{Subsec:TwoTrees}, this specifies a balanced presentation $P(T, X, X')$.

\begin{lemma}
\label{Lem:FromPToPTXX}
There is a sequence of at most $\ell$ moves of type $(4^+)$ taking $P'$ to $P(T,X,X')$. 
\end{lemma}

\begin{proof}
The 3-ball $B$ is a cone over $D$. The triangulation of $D$ has a dual graph and this contains a maximal tree $Y'$ that is obtained from $X'$ as follows. A vertex of $X'$ corresponds to a tetrahedron of $B$, which intersects $D$ in a triangle, and the corresponding vertex of $Y'$ arises from this triangle. Similarly, each edge of $X'$ corresponds to a face of $B$, and this intersects $D$ in a single edge, and the corresponding edge of $Y'$ is dual to this edge. When the interior of the faces of $D$ are removed and the edges corresponding to the edges of $Y'$ and the edges in $\partial D$ are removed from $D$, the result is a forest $F$. Let $C$ be the cone over this forest with cone point $v$. Then $K(T,X,X')$ is obtained from $K'$ by attaching $C$.
Each component of this forest $F$ intersects $\partial D$ in a single vertex. Hence, unless $F$ consists solely of vertices in $\partial D$, it has a leaf $w$ in the interior of $D$. This is attached to a single edge $e$ of $F$. The cone over $e$ is a triangle. We can remove this leaf $w$ and edge $e$ from $F$. The effect on the cone is to collapse the triangle via the edge joining $w$ to $v$. The reverse of this move is to add a new generator corresponding to this edge, as well as a relation that expresses this generator as a word of the length at most two in the previous generators. In other words, this is a type $(4^+)$ move. This is repeated for each vertex in the interior of $D$, of which there are at most $\ell$. 
\end{proof}

Note that the length of the presentation $P(T,X,X')$ is at most 3 times the number of faces of $T$, which is at most $6t$, and hence at most $12 \ell$.

\begin{proposition}
\label{Prop:PachnerMoveChangePres}
Suppose $(T, X, X')$ is a 3-sphere triangulation with auxiliary trees, and with $t$ tetrahedra. Let $T_2$ be a triangulation obtained from $T$ by a Pachner move. Then $T_2$ has auxiliary trees $X_2$ and $X_2'$ so that $P(T,X,X')$ and $P(T_2, X_2, X_2')$ differ by at most $12t^2+108t -13$ moves of type $(0)$-$(6)$. Moreover when each such move is applied, the length of the presentation is at most $6t+4$.
\end{proposition}

\begin{proof}
Consider first a 1-4 Pachner move. This introduces one vertex and four new edges. The maximal tree $X$ may be extended to a maximal tree $X_2$ in the 1-skeleton of $T_2$ by adding one of these edges. Let $w$ be this new edge of $X_2$, and let $x$, $y$ and $z$ be the three news edges of $T_2$ not lying in $X_2$. (See Figure \ref{Fig:PachnerAC}.) The dual graph for $T_2$ is obtained from that of $T$ by replacing a vertex by a copy of the complete graph $K_4$. Coming out of this vertex was 4 four edges (which need not be distinct), and each of these four edges becomes attached to one of the four new vertices. A maximal tree $X_2'$ in this new dual graph can be obtained from the previous one by adding in 3 edges in this copy of $K_4$. We can arrange that these 3 edges lie in a line in $K_4$, and hence the three new edges of the dual graph not lying in $X'_2$ also lie in a line. We can arrange that these new edges not lying in $X_2'$ correspond to the shaded faces in Figure \ref{Fig:PachnerAC}. Label some of the edges in the boundary of the original tetrahedron as shown in the figure. Hence, $P(T_2, X_2, X_2')$ is obtained from $P(T,X,X')$ by the following moves: 
\begin{enumerate}[(i)]
\item introduce the new generator $x$ and the new relation $x = a$; here $a$ might be the identity element or it might be a generator of $P(T,X,X')$ or its inverse, depending on whether $a$ lies in $X$ or it does not;
\item introduce the new generator $y$ and the new relation $y = b$;
\item introduce the new generator $z$ and the new relation $z = yc$.
\end{enumerate}
Thus, in this case, three $(4^+)$ moves are used. Note that the length of the presentation is initially at most $3$ times the number of faces, and hence at most $6t$. So before the final $(4^+)$ move is performed, the length is at most $6t+4$.

\begin{figure}[h]
\centering
\includegraphics[width=0.35\textwidth]{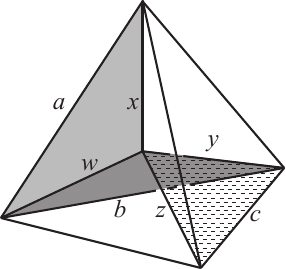}
\caption{The arrangement of tetrahedra after a 1-4 Pachner move}
\label{Fig:PachnerAC}
\end{figure}

Consider now a 4-1 Pachner move. The 1-skeleton of $T$ has four edges meeting at the vertex that is to be removed. As above, label these $w$, $x$, $y$ and $z$. Certainly at least one of these edges lies in $X$, $w$ say. We can actually arrange that the other three edges do not lie in $X$, as follows. Suppose, for instance, that $x$ lies in $X$. This implies that the 0-cells at the endpoints of $a$ are distinct. It also implies that $a$ does not lie in $X$, since $a$, $x$ and $w$ would form a cycle. Hence, we may perform an edge swap adding $a$ to $X$ and removing $x$. Repeat this if necessary for the other two edges $y$ and $z$. By Lemma \ref{Lem:EdgeSwap}, the number of moves of type (5) and (6) needed for each edge swap is at most $4(n-1) + 2\ell + \ell'$ where $n$ is the number of generators and $\ell$ and $\ell'$ are the lengths of the initial and final presentations. If $V$, $E$, $F$ denote the number of vertices, edges and faces of $T$, then $n = E - V + 1$. Also, $4t = 2F$ and $V - E + F - t = 0$. So $4(n-1) = 4(E-V) = 4(F-t) = 4t$. This is therefore an upper bound for the number of type (5) and (6) moves. Furthermore, when each move is applied, the length of the presentation is at most $6t$.

We now arrange that $X'_2$ is as above. Note that the dual graph contains a copy of $K_4$ corresponding to the four tetrahedra to be removed. In Figure \ref{Fig:PachnerAC}, the unshaded faces incident to the removed vertex form a line in this dual graph. Hence, by performing at most 3 edge swaps to $X'_2$, we may ensure that this line of edges lies in the maximal tree of the dual graph. Here, we are using Lemma \ref{Lem:EdgeSwapForest}. By Lemma \ref{Lem:EdgeSwapDualTree}, each edge swap has the effect of at most $4t^2+12t - 16$ Andrews-Curtis moves. After this, the presentation is as above. We can then perform (iii), (ii), (i) above in reverse, and obtain $P(T_2, X_2, X_2')$. We have used at most $12t^2 + 108t-13$ moves, and each time we applied them, the length of the presentation was at most $6t$.

Suppose now that we perform a 2-3 Pachner move. This introduces a new edge to the 1-skeleton of $T$. As it does not introduce any vertices, $X$ remains a maximal tree in the 1-skeleton, and so we set $X_2$ to be $X$. In the 2-3 Pachner move, a face is removed. This corresponds to an edge $e$ of the dual graph $X'$ joining distinct vertices of $X'$. By an edge swap, we can arrange that $e$ lies in $X'$. By Lemma \ref{Lem:EdgeSwapDualTree}, this uses at most $4t^2+12t-16$ Andrews-Curtis moves. In the dual graph after the Pachner move, there is a cycle of length 3 corresponding to the three new tetrahedra. We can obtain the new maximal tree $X'_2$ by using the same edges as in $X'$, apart from $e$, as well as two of the three edges in the cycle. Thus, $P(T_2, X_2, X_2')$ is obtained from the previous presentation by adding a new generator (corresponding to the new edge in the 1-skeleton) and a new relation that makes this generator equal to a word of length at most $2$ in the generators and their inverses. Thus, this is a single $(4^+)$-move.

Finally, consider a 3-2 Pachner move. This removes an edge $e$. We adjust the maximal tree $X$ if $e$ lies in $X$ as follows. Removing $e$ from $X$ creates a forest of two components. Let $x$ and $y$ be two edges in the boundary of the tetrahedra involved in the move, where $x$ is incident to one endpoint of $e$, $y$ is incident to the other endpoint of $e$, and $x$ and $y$ are incident. The path $x \cup y$ starts in one component of the forest and ends in the other. Hence, one of $x$ and $y$ has endpoints in the two different forest components, $x$ say. Adding $x$ to $X$ and removing $e$ forms the required edge swap. By Lemma \ref{Lem:EdgeSwap}, this is achieved using at most $4t$ type (5) and (6) moves.
In the dual 1-skeleton, there is the cycle of length three corresponding to the three tetrahedra to be removed. Using at most 2 edge swaps, we may ensure that two of these edges are in the maximal tree of the dual 1-skeleton. Thus, we are in the arrangement that appeared at the end of the 2-3 Pachner move considered above. Applying the 3-2 Pachner move is then an application of a $(4^+)$ move.
\end{proof}

\begin{proof}[Proof of Theorem \ref{Thm:ThickenableBound}]
We start with a balanced thickenable presentation $P$ for the trivial group with length $\ell$. Using at most $5\ell^2$ moves (0)-($4^+$), we convert it to a triangular balanced thickenable presentation $P'$ with at most $\ell$ generators and $\ell$ relations or to the standard one-generator presentation. In the latter case, the process terminates. Attaching a triangulated 3-ball $B$ as above, we obtain a triangulation $T$ of the 3-sphere with $t \leq 2\ell$ tetrahedra and with a single vertex. Let $X$ be this vertex, and let $X'$ be a maximal tree in the dual 1-skeleton of $B$. Then by Lemma \ref{Lem:FromPToPTXX}, there exists a sequence of at most $\ell$ moves of type ($4^+$) taking $P'$ to $P(T,X,X')$. As observed above, the length of $P(T,X,X')$ is at most $12 \ell$. Now apply Theorem \ref{Thm:PachnerS3} to obtain a sequence of at most 
$$M= a t^2 2^{bt^2} \leq 4 a \ell^2 2^{4 b \ell^2}$$
Pachner moves, where $a \leq 6 \cdot 10^6$ and $b \leq 5 \cdot 10^4$, taking $T$ to the standard 3-sphere triangulation. Since each Pachner move increases the number of tetrahedra by at most 3, all triangulations in this sequence have at most $3M$ tetrahedra (noting that, apart from $T$, at most $M-1$ Pachner moves relates any triangulation in the sequence to the standard one). By Proposition \ref{Prop:PachnerMoveChangePres}, after each Pachner move, there is a choice of auxiliary trees so that the resulting presentations differ by at most $108M^2 + 324 M -13$ moves (0)-(6). 

The final presentation arises from the standard triangulation of the 3-sphere. This has two vertices in its dual 1-skeleton, and so a maximal tree is a single edge. Thus, when we remove the interior of the face dual to this edge, as well as the interior of the two tetrahedra, the result is a disc comprised of three triangles. Depending on the maximal tree in the 1-skeleton and the orientations and basepoints on the faces, this gives one of the following presentations (up to relabelling the generators):
$$\langle x,y,z | x^{\pm 1}, y^{\pm 1}, z^{\pm 1} \rangle \qquad \langle x,y,z | x^{\pm 1}, y^{\pm 1}, z^{\pm1 } y^{\pm 1} \rangle \qquad
\langle x,y,z | x^{\pm 1}, y^{\pm 1}, y^{\pm1 } z^{\pm 1} \rangle$$
$$\langle x,y,z | x^{\pm 1}, y^{\pm 1}, x^{\pm 1} y^{\pm 1} z^{\pm 1} \rangle \qquad \langle x,y,z | x^{\pm 1}, x^{\pm 1} y^{\pm 1}, y^{\pm 1} z^{\pm1 }  \rangle.
$$
Each of these can be reduced to the standard one-generator presentation using at most 6 stable Andrews-Curtis moves.

So the total number of moves (0)-(6) is at most $5\ell^2 + \ell + 108M^3 + 324M^2 - 13M + 6$, which is at most $2^{kt^2}$
where $k \leq 7 \cdot 10^5$. Moreover, when each move is applied, the length of the presentation is at most $6M + 4\leq 2^{kt^2}$. 
\end{proof}

\section{An alternative approach using classical methods}
\label{Sec:AlternativeApproach}

As mentioned in the Introduction, the stable Andrews-Curtis conjecture for thickenable balanced presentations of the trivial group was well known. In this section, we give an outline of a proof using classical methods described in \cite{HogAngMetzler}. At the end, we speculate whether this alternative proof could be used to provide an explicit bound on the number of stable Andrews-Curtis moves.

We recall some of the terminology of \cite{HogAngMetzler, Rapaport}. A finite sequence of Andrews-Curtis moves is said to be a \emph{$Q$-transformation}. If, in addition, the Nielsen moves (5) and (6) are permitted, it is a \emph{$Q^\ast$-transformation}. If, furthermore, the stable Andrews-Curtis move (4) is allowed, it is a \emph{$Q^{\ast \ast}$-transformation}. Hence, by Lemma \ref{Lem:Nielsen} and Proposition \ref{Lem:Move4+}, a balanced presentation of the trivial group is related to the standard one-generator presentation by a $Q^{\ast \ast}$-transformation if and only if they differ by a finite sequence of stable Andrews-Curtis moves.
 
Recall that two finite cell complexes are \emph{simple homotopy equivalent} if they differ by a finite sequence of elementary expansions and collapses, which are defined as follows. Suppose that a cell complex $K'$ is obtained from a cell complex $K$ by attaching an $(n-1)$-cell $c_{n-1}$ followed by an $n$-cell $c_n$. So $K \cup c_{n-1}$ is obtained as a quotient of the disjoint union of $K$ and $D^{n-1}$. The restriction of the quotient map to $D^{n-1}$ is the \emph{characteristic map} of $c_{n-1}$. Suppose that the attaching map $\partial D^n \rightarrow K \cup c_{n-1}$ for $c_n$ satisfies the following:
\begin{enumerate}
\item the restriction of the map to the upper hemisphere of $\partial D^n$ is equal to the standard identification between the upper hemisphere and $D^{n-1}$ composed with the characteristic map of $c_{n-1}$;
\item the image of the lower hemisphere of $\partial D^n$ lies in $K$.
\end{enumerate}
Then $K$ is obtained from $K'$ by an \emph{elementary collapse}, and $K'$ is obtained from $K$ by an \emph{elementary expansion}. Two cell complexes are \emph{$n$-deformation equivalent} if they differ by a sequence of elementary expansions and collapses, where the cells involved have dimension at most $n$.
 
A central observation \cite{Wright} (see also \cite[Theorem 2.4]{HogAngMetzler}) is the following.
 
\begin{lemma}
\label{Lem:3Deformation}
Let $P$ and $P'$ be two finite group presentations. Then they differ by a $Q^{\ast \ast}$-transformation if and only if their presentation 2-complexes are $3$-deformation equivalent.
\end{lemma}
 
Thus, an alternative proof that any thickenable balanced presentation $P$ of the trivial group can be reduced to the standard one-generator presentation $P'$ via stable Andrews-Curtis moves is as follows. Let $K$ be the presentation 2-complex of $P$, and let $N(K)$ be its 3-manifold regular neighbourhood. Then $N(K)$ (with a suitable cell structure) collapses to $K$, via 3-deformation collapses. Let $K'$ be the presentation complex associated with the standard one-generator presentation of the trivial group; topologically, this is just a disc. We may embed $K'$ in the interior of $N(K)$, and then $N(K)$ is also a regular neighbourhood of $K'$. Thus, we are in the setting of the following proposition. This is stated in \cite[Chapter I, Section 2]{HogAngMetzler} but a full proof is not given.
 
\begin{proposition}
Let $M$ be a compact 3-manifold containing 2-complexes $K$ and $K'$ such that $M$ is a regular neighbourhood of $K$ and a regular neighbourhood of $K'$. Then $K$ and $K'$ are $3$-deformation equivalent.
\end{proposition}
 
Hence, by Lemma \ref{Lem:3Deformation}, $P$ and $P'$ differ by a $Q^{\ast \ast}$-transformation, and so differ by a sequence of stable Andrews-Curtis moves.
 
\subsection{Controlling the number of moves}
 
To provide a concrete bound on the number of stable Andrews-Curtis moves relating $P$ and $P'$, as in Theorem \ref{Thm:StableACBound}, we would need to quantify the above argument, in three ways:
\begin{enumerate}
\item We would need to convert a $Q^{\ast \ast}$-transformation into an explicit sequence of stable Andrews-Curtis moves.
\item We would need to convert an elementary expansion or collapse, using cells of dimension at most 3, into a $Q^{\ast \ast}$-transformation.
\item Although we have elementary collapses from $N(K)$ to $K$ and from $N(K')$ to $K'$, this does not immediately provide a sequence of 3-deformation elementary expansions and collapses relating $K$ and $K'$, since we need to take account of the fact that $N(K)$ and $N(K')$ have different cell structures.
\end{enumerate}
The method of dealing with (1) is developed in Section 2. It seems straightforward to handle (2). However, it is (3) that seems most challenging. To relate two different cell structures for a compact 3-manifold is at the heart of the homeomorphism problem for compact 3-manifolds. The homeomorphism problem for compact orientable 3-manifolds is solved \cite{Kuperberg, ScottShort}, and it is well known that this is equivalent to the existence of a computable upper bound on the number of Pachner moves needed to relate two triangulations of a compact orientable 3-manifold (see \cite[Theorem 3.1]{Lackenby} for instance). Indeed, Mijatovic's Theorem \ref{Thm:PachnerS3} was based on the Rubinstein-Thompson algorithm for 3-sphere recognition \cite{Rubinstein, Thompson}. Thus, although the methods developed in \cite{HogAngMetzler} do lead to an alternative proof that thickenable balanced presentations of the trivial group can be trivialised via stable Andrews-Curtis moves, they do not easily lead to an explicit bound on the number of these moves, as in Theorem \ref{Thm:StableACBound}. Indeed, to do so would seem to require a solution to the recognition problem for the 3-sphere, which underpins Mijatovic's and King's theorems \cite{Mijatovic, King}, and hence it would probably need to follow some version of the techniques presented in this paper.

\section{The Andrews-Curtis conjecture}
\label{Sec:ACConjecture}

Our aim in this section is to prove Theorem \ref{Thm:AC}. As explained in the Introduction, the proof is primarily due to Guo \cite{Guo}, who in turned relied on a theorem about Heegaard splittings of the 3-sphere 
due to Waldausen \cite{Waldhausen}. We therefore now recall some terminology about Heegaard splittings.

A \emph{Heegaard splitting} for a closed orientable 3-manifold $M$ is a description of $M$ as the union of two handlebodies $H_1$ and $H_2$ glued via a homeomorphism between their boundaries. We view $H_1$ and $H_2$ as subsets of $M$, and write this splitting as $(H_1, H_2)$. The \emph{genus} of the splitting is equal to the genus of the two handlebodies. We say that two Heegaard splittings $(H_1, H_2)$ and $(H'_1, H'_2)$ of $M$ are \emph{equivalent} if  there is an orientation-preserving homeomorphism of $M$ taking $H_1$ to $H_1'$ and taking $H_2$ to $H'_2$. The 3-sphere admits a \emph{standard Heegaard splitting} of any genus $g$, obtained by taking $H_1$ to be a regular neighbourhood of a planar embedding of a wedge of $g$ circles, and letting $H_2$ be the closure of the complement of $H_1$. Waldhausen's theorem is as follows.

\begin{theorem}
\label{Thm:Waldhausen}
Any Heegaard splitting of 3-sphere is equivalent to a standard one.
\end{theorem}

In a handlebody $H$ of genus $g$, there is a collection of $g$ disjoint properly embedded discs $D_1, \dots, D_g$ such that the closure of $H \setminus (D_1 \cup \dots \cup D_g)$ is a 3-ball. Such a collection of discs is called a \emph{meridional collection}, and the discs are called \emph{meridian discs}. Suppose that, for a Heegaard splitting $(H_1, H_2)$, meridional collections $\mathcal{D}_1$ and $\mathcal{D}_2$ are chosen for $H_1$ and $H_2$ respectively. Then $D_i \in \mathcal{D}_1$ and $D'_j \in \mathcal{D}_2$ are a \emph{destabilising pair} if $D_i$ intersects $D'_j$ transversely precisely once and misses all other discs in $\mathcal{D}_2$ and similarly $D'_j$ misses all other discs in $\mathcal{D}_1$. Then one can form a new Heegaard splitting with genus that is 1 smaller. One of the handlebodies is isotopic to the closure of $H_1 \setminus N(D_i)$.

The standard Heegaard splitting $(H_1, H_2)$ of the 3-sphere with genus $g$ has meridional collections $\{ D_1, \dots, D_g \}$ and $\{ D'_1, \dots, D'_g \}$ such that $D_i$ and $D'_i$ form a destabilising pair for each $i$. We call this the \emph{standard pair} of meridional collections for this splitting.

When meridional collections $\mathcal{D}_1$ and $\mathcal{D}_2$ are chosen for a Heegaard splitting $(H_1, H_2)$ of a 3-manifold $M$, this determines a presentation for $\pi_1(M)$, once a small amount of extra data is chosen. If we pick a transverse orientation on each disc in $\mathcal{D}_1$, then $\pi_1(H_1)$ becomes identified with a free group on $g$ generators. The meridian discs $\mathcal{D}_2$ determine relations, once the boundary of each of these discs is oriented and a basepoint is chosen for each one. For example, with suitable choices, the standard pair of meridional collections specifies the standard presentation $\langle x_1, \dots, x_g | x_1, \dots, x_g \rangle$ for the trivial group.

It is possible to modify a meridional collection of discs for a handlebody $H$ to a new meridional collection $D_1, D_2, \dots, D_{i-1}, D_i', D_{i+1}, \dots D_g$, via a \emph{disc slide} which is defined as follows. Pick an embedded arc $\alpha$ in $\partial H$ joining $\partial D_i$ to the boundary of some other $D_j$, but which is otherwise disjoint from $D_1 \cup \dots \cup D_g$. Let $N$ be a regular neighbourhood of $D_i \cup \alpha \cup D_j$ in $H$. Then the closure of $\partial N \setminus \partial H$ consists of three discs. Two of these are parallel to $D_i$ and $D_j$. Let $D'_i$ be the third disc.

The following is simple and well known \cite{Singer} (see also \cite[Theorem 1]{Wajnryb}).

\begin{lemma}
\label{Lem:DiscSlides}
Any two meridional collection of discs for a handlebody differ by a sequence of disc slides and ambient isotopies.
\end{lemma}

The following result is essentially the argument given by Guo \cite{Guo}.

\begin{theorem}
\label{Thm:Qast}
Let $P$ be a thickenable balanced presentation of the trivial group. Then $P$ is related to a standard presentation by a $Q^{\ast}$-transformation.
\end{theorem}

\begin{proof}
Let $P$ be a thickenable balanced presentation of the trivial group, and let $K$ be its presentation 2-complex. Let $N(K)$ be its 3-manifold regular neighbourhood. Attach on a 3-ball to $\partial N(K)$, forming a 3-sphere $M$. This has a natural Heegaard splitting, where one handlebody $H_1$ is the union of the 0-handle and 1-handles of $N(K)$, and where $H_2$ is the closure of the complement of $H_1$. Let $g$ be its genus. It also has a natural meridional collection of discs for each handlebody. For $H_1$, this collection $\mathcal{D}_1$ is the set of co-cores of the 1-handles. For $H_2$, the collection $\mathcal{D}_2$ is the set of cores of the 2-handles. Then, when suitable orientations and basepoints are chosen, the presentation for $\pi_1(M)$ arising from $\mathcal{D}_1$ and $\mathcal{D}_2$ is $P$.

By Waldhausen's Theorem \ref{Thm:Waldhausen}, this Heegaard splitting is equivalent to the standard one. Hence, by Lemma \ref{Lem:DiscSlides}, there is a sequence of disc slides that one can apply to $\mathcal{D}_1$ and $\mathcal{D}_2$ taking them to the standard pair of meridional systems.

A disc slide to the discs in $\mathcal{D}_2$ has the following effect on the presentation $\langle x_1, \dots, x_g | r_1, \dots, r_g \rangle$ arising from $\mathcal{D}_1$ and $\mathcal{D}_2$. Suppose that the arc $\alpha$ along which we slide has endpoints on $D_i$ and $D_j$ that are equal to the basepoints of these discs. This can be achieved by changing these basepoints. A change of basepoints to a disc in $\mathcal{D}_2$ induces a cyclic permutation of the corresponding relation and this achieved by a sequence of type (0) and (3) Andrews-Curtis moves. Then the disc slide replaces $r_i$ by the product of $r_i$ and a conjugate of $r_j$ or its inverse, because the boundary of the new disc runs around $\partial D_i$, then along $\alpha$, then around $\partial D_j$ and then back along $\alpha$. Finally, a change of orientation to a disc in $\mathcal{D}_2$ changes the corresponding relation by a type (1) move. Hence, the effect of these modifications to the meridional discs in $H_2$ is achieved by Andrews-Curtis moves.

A disc slide to the discs in $\mathcal{D}_1$ has the effect of changing the generators by a Nielsen move (5). A change of orientation to one of these discs  replaces the corresponding generator by its inverse, which is Nielsen move (6).

Thus, the presentation $P$ is indeed related to the standard $g$ generator presentation of the trivial group by a sequence of Andrews-Curtis moves (0)-(3) and Nielsen moves (5) and (6).
\end{proof}

To complete the proof of Theorem \ref{Thm:AC}, we need to show how to eliminate the use of Nielsen moves. Guo \cite{Guo} does not explain how to do this. But it is achieved by the following result, which, to the author's knowledge, is new.

\begin{theorem}
\label{Thm:QastToAC}
Let $P$ be a balanced presentation of the trivial group. If $P$ is related to a standard presentation by a $Q^{\ast}$-transformation, then these presentations are related by Andrews-Curtis moves.
\end{theorem}

We now embark upon the proof of this theorem. To do so, it is convenient to work with a slightly different notion of a group presentation than elsewhere in the paper. Previously, we have viewed the relations in a presentation $\langle x_1, \dots, x_n | r_1, \dots, r_n \rangle$ as elements of the free semi-group on the generators and their inverses. But we now instead view the relations as elements of the free group $\langle x_1, \dots, x_n \rangle$. This is a modest change. It has the advantage that the Andrews-Curtis move (0), which introduces or removes a generator and its inverse, is no longer required. It also has the advantage that the Nielsen moves (5) and (6) are realised by actual automorphisms of the free group $\langle x_1, \dots, x_n \rangle$ rather than some lift of such an automorphism to the free semi-group. However, it has the disadvantage that thickenability is not immediately well-defined, since it might be possible that one semi-group realisation of the relations is thickenable whereas another is not. However, this does not concern us, 
since we are now focused on Theorem \ref{Thm:QastToAC}, which does not refer to thickenability.

\begin{lemma}
\label{Lem:ChangeOrder}
Let $P$ be a finite group presentation, and let $P'$ be obtained from $P$ by a Nielsen transformation followed by an Andrews-Curtis move. Then $P'$ is also obtained from $P$ by an Andrews-Curtis move followed by a Nielsen transformation. Similarly, if $P'$ is obtained from $P$ by an Andrews-Curtis move followed by a Nielsen transformation, we may also interchange the order of those moves.
\end{lemma}

\begin{proof}
Suppose that we have a Nielsen transformation followed by an Andrews-Curtis move. As explained in Section \ref{Sec:Nielsen}, the Nielsen transformation is achieved by an automorphism $\phi$ of the free group $\langle x_1, \dots, x_n \rangle$. This takes the relations $\{ r_1, \dots, r_n \}$ to $\{ \phi(r_1), \dots, \phi(r_n) \}$. Say that this is followed by Andrews-Curtis move (1), exchanging $\phi(r_i)$ for $(\phi(r_i))^{-1}$. Then we may instead replace $r_i$ by $r_i^{-1}$, and then apply $\phi$. A similar argument works for moves (2) and (3). The proof when the moves are given to us in the opposite order is essentially the same.
\end{proof}

\begin{proof}[Proof of Theorem \ref{Thm:QastToAC}]
We are assuming that the presentation $P$ is related to a standard presentation of the trivial group by a $Q^{\ast}$-transformation, in other words by a sequence of Andrews-Curtis moves and Nielsen moves.
By Lemma \ref{Lem:ChangeOrder}, in this sequence of moves, we may perform all the Andrews-Curtis moves first, then all the Nielsen moves. The final presentation is $\langle x_1, \dots, x_n | x_1, \dots, x_n \rangle$. Hence, before the Nielsen moves, the presentation $P'$ is $\langle x_1, \dots, x_n | \phi(x_1), \dots, \phi(x_n) \rangle$ for some automorphism $\phi$ of the free group. This automorphism is expressed as a composition of Nielsen transformations. Say that there are $k$ of these. Consider the last Nielsen transformation in the sequence. This converts a presentation $P''$ to the standard one. Hence, $P''$ is obtained from a standard presentation using a Nielsen move. But, when applied to a standard presentation, this is the same as applying an Andrews-Curtis move. Thus, we have a sequence of $k-1$ Nielsen transformations taking $P'$ to $P''$, followed by an Andrews-Curtis move. Now apply Lemma \ref{Lem:ChangeOrder} to replace this by an Andrews-Curtis move followed by $k-1$ Nielsen moves. Thus, we have reduced the number of Nielsen moves, and continuing in this way, we eventually remove all Nielsen moves.
\end{proof}

This completes the proof of Theorem \ref{Thm:AC}.

We note that instead of using Theorem \ref{Thm:QastToAC}, we could have completed the proof of Theorem \ref{Thm:AC} by applying the following theorem of Kaneto \cite{Kaneto}, which is of independent interest.

\begin{theorem}
\label{Thm:Kaneto}
Let $\mathcal{D}_1$ and $\mathcal{D}_2$ be a pair of meridional collections of discs for the standard Heegaard splitting of the 3-sphere. Then there is a sequence of disc slides taking $\mathcal{D}_2$ to a new collection $\mathcal{D}_2'$, so that when $\mathcal{D}_2'$ is suitably oriented, the group presentation arising from $\mathcal{D}_1$ and $\mathcal{D}_2'$ is related to a standard presentation of the trivial group using only cancellation moves of type $(0)$.
\end{theorem}

Thus, disc slides to the discs in $\mathcal{D}_1$ are not actually required, and it was these that necessitated the use of the Nielsen moves (5) and (6). Hence, Guo's proof combined with Kaneto's theorem would complete the proof of Theorem \ref{Thm:AC}. However, we chose to present the proof as above, because of the very general nature of Theorem \ref{Thm:QastToAC}.

\subsection{An example.} We present an illustrative example.
Let $H_1$ be the handlebody embedded in the 3-sphere shown in the left of Figure \ref{Fig:HS3Sphere}. This is ambient isotopic to the handlebody shown on the right in Figure \ref{Fig:HS3Sphere}, and the closure of the complement of this is also a handlebody. Hence, the closure of the complement of $H_1$ is a handlebody $H_2$. There are two meridian discs for $H_1$, denoted $D_1$ and $D_2$ in Figure \ref{Fig:HS3Sphere}. The meridian discs $D_1'$ and $D_2'$ shown in the right of Figure \ref{Fig:HS3Sphere} pull back to meridian discs $D_1''$ and $D_2''$ for $H_2$.

\begin{figure}[h]
\centering
\includegraphics[width=0.95\textwidth]{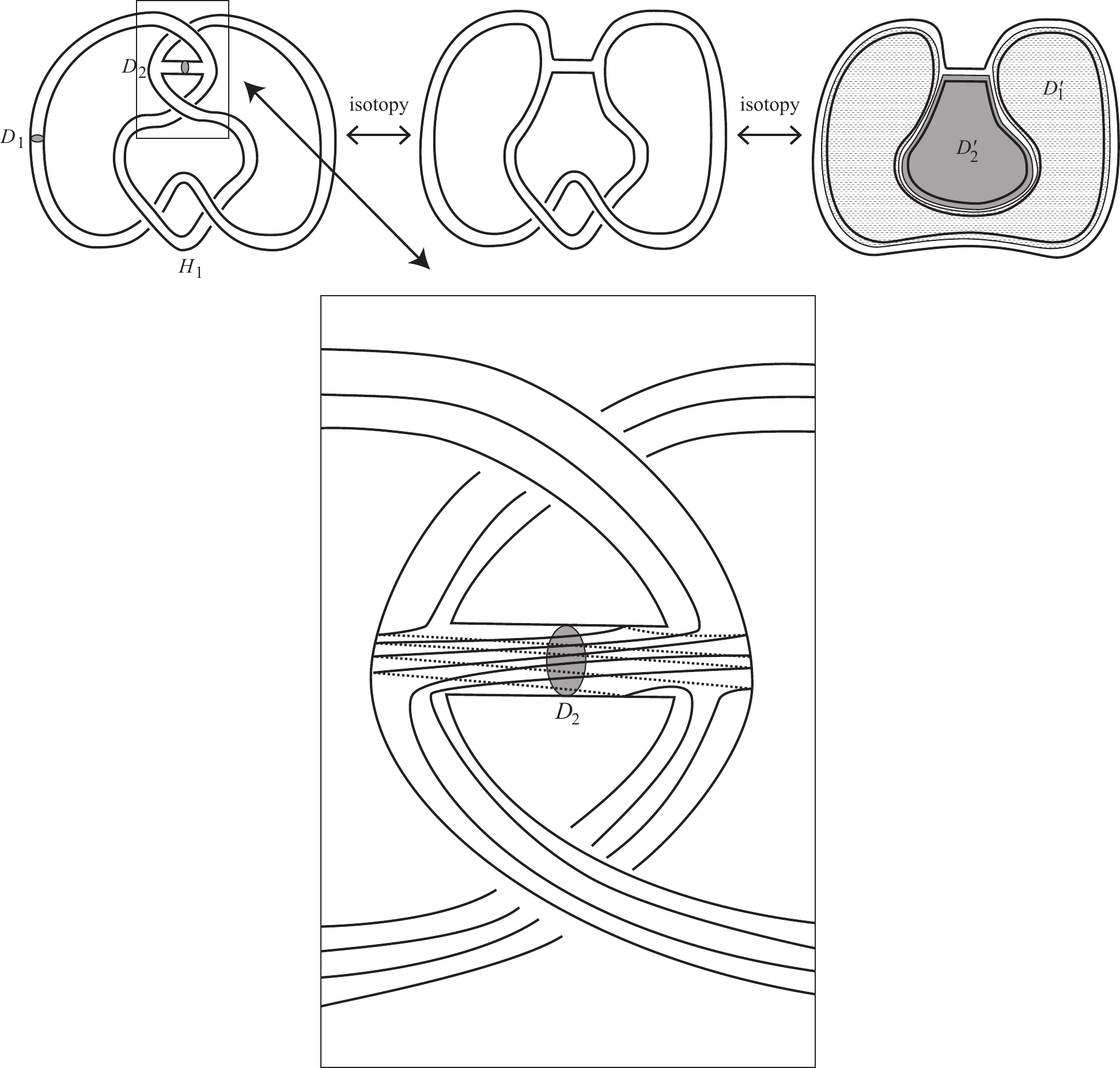}
\caption{Left: a Heegaard splitting of the 3-sphere, Middle and right: isotopic splittings, Right: Meridian discs $D_1'$ and $D'_2$, Bottom: the boundaries of the corresponding discs $D''_1$ and $D''_2$ in the original splitting}
\label{Fig:HS3Sphere}
\end{figure}

Although this is by, Waldhausen's theorem, equivalent to the standard Heegaard splitting for the 3-sphere, we have the following result.

\begin{claim}
The meridional collection $\{ D_1,  D_2 \}$ is not part of standard pair of meridional collections for the Heegaard splitting.
\end{claim}

Suppose that there were meridian discs $\tilde D_1$ and $\tilde D_2$ for $H_2$ such that $\{ D_1, D_2 \}$ and $\{ \tilde D_1, \tilde D_2 \}$ form the standard pair. Then we could destabilise the Heegaard splitting, decreasing its genus by $1$, by setting $\tilde H_1$ to be the closure of $H_1 \setminus N(D_2)$. However, the closure of the complement of $N(D_2)$ is not a handlebody in a Heegaard splitting of the 3-sphere. This is because $\tilde H_1$ is a knotted solid torus. This proves the claim.

Thus, it is not possible, by using only disc slides in $H_2$, to turn the initial meridional systems into the standard pair. Nevertheless, the initial meridional systems $\{ D_1, D_2 \}$ and $\{ D_1'', D_2'' \}$ give the presentation $\langle x_1 , x_2 | x_1 x_2^{-1} x_2^2 x_2^{-1}, x_2 x_2^{-1} x_2 x_2^{-1} x_2 \rangle$. Four type (0) cancellation moves take this to a standard presentation. This is consistent with Kaneto's theorem. 

\subsection{Bounding the number of Andrews-Curtis moves.} 
Unfortunately, we do not obtain an explicit upper bound on the number of Andrews-Curtis moves in Theorem \ref{Thm:AC}. 
This is because there is no known explicit bound on the number of handle slides required to relate meridional discs systems for a Heegaard splitting of the 3-sphere to the standard meridional disc systems. Indeed, this is related to the classical problem, going back to Whitehead \cite{Whitehead}, of how to convert a Heegaard diagram for the 3-sphere into a standard one. A good deal of work by many researchers (for example \cite{Volodin, Viro, Ochiai, Morikawa}) has shown why this problem is challenging. At present, all that we can say is that there is a computable upper bound on the number of Andrews-Curtis moves, as function of the length of the presentation. Indeed, a naive computation for this quantity runs as follows. Consider all balanced group presentations with length $\ell$ (and with no relations forming empty words). For each such presentation $P$, determine whether it is thickenable, using the algorithm of Neuwirth \cite{Neuwirth}. Discard all that are not thickenable. This algorithm also provides a thickening $N(K)$. The presentation $P$ was a presentation of the trivial group if and only if $N(K)$ is a 3-ball, and there is an algorithm to determine this \cite {Rubinstein, Thompson}. Hence, we can list all balanced thickenable presentations of the trivial group with length $\ell$. We can then start to apply Andrews-Curtis moves to them. By Theorem \ref{Thm:AC}, we will eventually convert them all to standard presentations. The number of moves that we need is then the required computable function.

It would be very interesting to develop the theory of Heegaard splittings sufficiently in order to obtain a more explicit upper bound on the number of Andrews-Curtis moves for thickenable balanced presentations of the trivial group.

\bibliography{StableAC-biblio}
\bibliographystyle{plain}

\end{document}